\documentclass[12pt,amsfonts,twoside]{article}
\DeclareMathAlphabet{\eusm}{OT1}{eusm}{m}{n}
\DeclareMathAlphabet{\mathsl}{OT1}{cmr}{m}{sl}
\usepackage{amsmath}
\usepackage[all]{xy}
\usepackage{amssymb}
\usepackage{amscd}
\usepackage{amsmath, pb-diagram}
\usepackage{enumerate}
\newtheorem{thm}{Theorem}[section]
\newtheorem{lem}[thm]{Lemma}
\newtheorem{prop}[thm]{Proposition}
\newtheorem{cor}[thm]{Corollary}
\newtheorem{defn}[thm]{Definition}

\newtheorem{rem}[thm]{Remark}
\newtheorem{example}[thm]{Example}

\newenvironment{proof}{\par\noindent{\bf Proof \,}}{$\hfill
\Box$\par\bigskip}

\title {A Generalization of Injective\\ and Projective Complexes}
\author{Tahire \"Ozen, Emine Yildirim}

\date{}
\begin{document}
\maketitle \abstract{In this paper it is given a generalization of projective and injective complexes.}

\section{Introduction}

\begin{defn} \label{1.1} ($\mathcal{X}-complex$ ) Let $\mathcal{X}$ be a class of R-modules. A
complex $\mathcal{C}: \ldots \longrightarrow C^{n-1} \longrightarrow
C^{n} \longrightarrow C^{n+1} \longrightarrow \ldots$ is called an
$\mathcal{X}-(cochain)$ complex, if  $C^i \in \mathcal{X}$ for all
$i \in \mathbb{Z}$. A complex $\mathcal{C}: \ldots \longrightarrow
C_{n+1} \longrightarrow C_{n} \longrightarrow C_{n-1}
\longrightarrow \ldots$ is called an $\mathcal{X}-(chain)$ complex,
if $C_i \in \mathcal{X}$ for all $i\in \mathbb{Z}$. The class of all
$\mathcal{X}-complexes$ is denoted by
$C(\mathcal{X^{*}})$.($\mathcal{X^{*}}=\mathcal{X}-complex$).
\end{defn}

\begin{defn} \label{1.2} ($\mathcal{X}-injective$ and
$\mathcal{X}-projective$ module) A left R-module E is
$\mathcal{X}-injective$, if $Ext^{1}(B/A,E)=0$ where for every
module $B/A \in \mathcal{X}$, whenever i is an injection and $f$ is
any map, there exists a map
$g$ making the following diagram commute; \\

\[\begin{diagram}
\node[2]{E}\\
\node{0} \arrow{e} \node{A} \arrow{n,t} {f} \arrow{e,r}{i} \node{B}
\arrow{nw,t,..} {g}
\end{diagram}\]

\noindent Let a left $R$-module $P$ is projective if,
$Ext^{1}(P,X)=0$ for every module $X \in \mathcal{X}$ where
$X=Ker{p}$, whenever $p$ is surjective and $h$ is any map, there
exists a map $g$ making the following diagram
commute; \\

\[\begin{diagram}
\node[2]{P} \arrow{sw,t,..} {g}\arrow{s,r}{h} \\
\node{A} \arrow{e,r}{p} \node{B} \arrow{e} \node{0}
\end{diagram}\] \\

\end{defn}
\begin{defn} \label{1.21} ($\mathcal{X}$-precover, $\mathcal{X}-preenvelope$ )
Let $\mathcal{A}$ be an abelian category and let $\mathcal{X}$ be a
class of objects of $\mathcal{A}$. Then for an object $M \in
\mathcal{A}$, a morphism $\phi: X \longrightarrow M$ where $X \in
\mathcal{X}$ is called an $\mathcal{X}$-precover of M, if $f: X'
\longrightarrow M$ where $X' \in \mathcal{X}$ and the following
diagram;

\[\begin{diagram}
\node{X'} \arrow{e,t}{f} \arrow{s,t,..}{g} \node{M} \\ \node{X}
\arrow{ne,r}{\phi}
\end{diagram}\] \\

\noindent can be completed to a commutative diagram.

Let $\mathcal{A}$ be an abelian category and let $\mathcal{X}$ be a
class of objects of $\mathcal{A}$. Then for an object \;$M \in
\mathcal{A}$\;a morphism \;$\phi: M \longrightarrow X$\; where $X
\in \mathcal{X}$ is called an \;$\mathcal{X}$-preenvelope of $M$, if
\;$f:M \longrightarrow  X'$\; where \;$X' \in \mathcal{X}$\; and the
following diagram;

\[\begin{diagram}
\node{M} \arrow{e,t}{f} \arrow{s,t}{\phi} \node{X'}\\
\node{X} \arrow{ne,r,..}{g}
\end{diagram}\] \\
\noindent can be completed to a commutative diagram.
\end{defn}

\noindent Every complex has an injective and projective resolution. In this paper 
Ext is defined the same as the class of R-modules.(See \cite{3})

\section{$\mathcal{X}-injective$ and $\mathcal{X}-projective$ complexes}

\begin{defn} \label{1.3} ($\mathcal{X}-injective$ and $\mathcal{X}-projective$ complexes) A
complex $\mathcal{C}$ is called an $\mathcal{X}-injective$ complex,
if $Ext^{1}(Y/X,C) = 0$ where for every complex $Y/X \in
C(\mathcal{X})$. Thus the following diagram commutes as follows;

\[\begin{diagram}
\node[2]{C}\\
\node{0} \arrow{e} \node{X} \arrow{n,t} {f} \arrow{e,r}{\phi}
\node{Y} \arrow{nw,t,..} {\widetilde{f}}
\end{diagram}\]

\noindent such that ${\widetilde{f}}{\phi}={f}$  where ${\phi}$ is a
monomorphism.

Dually we can define an $\mathcal{X}-projective$ complex. A complex
$\mathcal{C}$ is called an $\mathcal{X}-projective$ complex, if
$Ext^{1}(C,X) = 0$ for every complex $X \in \mathcal{C(X)}$ where
$Ker{\phi}=X \in \mathcal{C(X)}$. By diagram we can show as follows;

\[\begin{diagram}
\node[2]{C} \arrow{sw,t,..} {\widetilde{f}}\arrow{s,r}{f} \\
\node{A} \arrow{e,r}{\phi} \node{B} \arrow{e} \node{0}
\end{diagram}\]

\noindent such that ${\phi}{\widetilde{f}}={f}$ where ${\phi}$ is
onto.
\end{defn}

\begin{example} \label{1.4} If $P$ is an $\mathcal{X}-projetive(\mathcal{X}-injective)$ module, then
$\overline{P}:... \longrightarrow 0 \longrightarrow P
\longrightarrow P \longrightarrow 0 \longrightarrow 0
\longrightarrow ...$ is an
$\mathcal{X}-projective(\mathcal{X}-injective)$ complex. Also the
finite direct sum of $\mathcal{X}-projective(\mathcal{X}-injective)$
complexes is again $\mathcal{X}-projective(\mathcal{X}-injective)$.

Note that if $P$ is an
$\mathcal{X}$-injective($\mathcal{X}$-projective) module and $P$ is
not in the class $\mathcal{X}$, then $\overline{P}$ is an
$\mathcal{X}$-injective complex, but not an $\mathcal{X}$-complex.
So $\mathcal{X}$-injective complex may not be an
$\mathcal{X}$-complex.

\end{example}

\begin{proof}Consider the following diagram;

\[\begin{diagram}
\node[2]{P} \arrow{s,r}{f_{k}} \arrow{se,t}{1_p} \arrow{sw,t,..}{h_{k}} \\
\node{C_{k}} \arrow{e,r}{g_{k}} \arrow{se,r}{d_{k}} \node{C'_{k}}
\arrow{se,r}{d'_{k}} \node{P} \arrow{s,r}{f_{k-1}} \arrow{sw,t,..}{h_{k-1}} \\
\node[2]{C_{k-1}}  \arrow{e,r}{g_{k-1}} \node{C'_{k-1}}
\end{diagram}\]

\noindent where $g:C \longrightarrow C' \longrightarrow 0$ is an
epic morphism such that $Kerg \in \mathcal{X}$  and $f:\overline{P}
\longrightarrow C$ is a morphism and d, $d'$ are differentials of
$C$, $C'$, respectively. Since $P$ is an $\mathcal{X}-projective$
module, there exists a morphism $h_{k}:P \longrightarrow C_{k}$ with
${g_{k}}{h_{k}}={f_{k}}$.

Define $h_{k-1}:P \longrightarrow C_{k-1}$ with
$h_{k-1}={d_{k}}{h_{k}}$. Then ${g_{k-1}}{h_{k-1}}={f_{k-1}}$. Thus
we have a chain map $h:\overline{P} \longrightarrow C$, as required.
Dually, we can prove that if $P$ is an $\mathcal{X}-injective$
module, then $\overline{P}$ is an $\mathcal{X}-injective$ complex.
\end{proof}

\begin{defn} \label{1.5} (DG($\mathcal{X}$-injective) and DG($\mathcal{X}$-projective) complexes)
Let $\varepsilon$ be the class of exact complexes. A complex I is
called DG($\mathcal{X}$-injective), if each $I^{n}$ is
$\mathcal{X}$-injective and $\mathcal{H}om(E,I)$ is exact for all $E
\in \varepsilon$ where $d^{n}:E^{n} \longrightarrow E ^{n+1}$ with
$Kerd^{n} \in \mathcal{X}$.

A complex I is called $DG(\mathcal{X}-projective)$, if each $I^{n}$
is $\mathcal{X}-projective$ and $\mathcal{H}om(I,E)$ is exact for
all $E \in \varepsilon$ where $d^{n}:E^{n} \longrightarrow E^{n+1}$
with $Kerd^{n} \in \mathcal{X}$.
\end{defn}

\begin{lem} \label{1.6} Let $A \overset{\beta} \longrightarrow B \overset{\theta} \longrightarrow
C$ be an exact sequence where $Ker \theta \in \mathcal{X}$. Then for
all $\mathcal{X}-projective$ complex I,
$$Hom(I,A) \longrightarrow
Hom(I,B) \longrightarrow Hom(I,C)$$ is exact.
\end{lem}

\begin{proof} $0 \longrightarrow Ker{\theta} \overset{\beta} \longrightarrow B \overset{\theta} \longrightarrow
C$ is exact, so $$0 \longrightarrow Hom(I,Ker{\theta})
\longrightarrow Hom(I,B) \longrightarrow Hom(I,C)$$ is exact.

We have the following commutative diagram;

\[\begin{diagram} \node{A} \arrow{e,t} {\beta} \node{Im{\beta}} \arrow{e}
\node{0} \\ \node[2]{I} \arrow{n,t} {g} \arrow{nw,r,..} {f} \\
\end{diagram}\]

\noindent such that ${\beta}{f}={g}$.

\noindent Since I is $\mathcal{X}$-projective and $Ker \theta \in
\mathcal{X}$.

Therefore, $Hom(I,A) \longrightarrow Hom(I,B) \longrightarrow
Hom(I,C)$ is exact.
\end{proof}

\noindent Dually we can give the following lemma;

\begin{lem} \label{1.7} Let $A \overset{\beta} \longrightarrow B \overset{\theta} \longrightarrow
C$ be an exact sequence where $\frac{C}{Im{\theta}} \in
\mathcal{X}$. Then for all $\mathcal{X}-injective$ complex I,
$$Hom(C,I) \longrightarrow
Hom(B,I) \longrightarrow Hom(A,I)$$ is exact.
\end{lem}

\begin{proof} It is clear from Lemma \ref{1.6}.
\end{proof}

\begin{example} \label{1.8} (DG($\mathcal{X}-injective$)(DG($\mathcal{X}-projective$)))
Let $I=.... \longrightarrow 0 \longrightarrow I^{0} \longrightarrow
0 \longrightarrow 0 \longrightarrow ...$ where $I^{0}$ is an
$\mathcal{X}$-injective($\mathcal{X}-projective$) module. Then I is
DG($\mathcal{X}-injective$)(DG($\mathcal{X}-projective$)) complex.
\end{example}

\begin{proof} Let $$E:... \longrightarrow E^{-1} \overset{d^{-1}} \longrightarrow E^{0} \overset{d^{0}} \longrightarrow
E^{1} \overset{d^{1}} \longrightarrow E^{2} \overset{d^{2}}
\longrightarrow E^{3} \longrightarrow...$$ is exact and $Kerd^{n}
\in \mathcal{X}$, then

$$\mathcal{H}om(E,I) \cong ... \longrightarrow Hom(E^{2},I^{0}) \longrightarrow Hom(E^{1},I^{0})
\longrightarrow Hom(E^{0},I^{0}) \longrightarrow ...$$.\\
\noindent By Lemma \ref {1.6} $\mathcal{H}om(E,I)$ is exact.
\end{proof}

\begin{lem} \label{1.9} If a complex $X: \ldots \longrightarrow X_{n+1} \longrightarrow X_{n} \longrightarrow
X_{n-1} \longrightarrow \ldots$ is an
$\mathcal{X}-injective$($\mathcal{X}-projective$) complex, then for
all $n \in \mathbb{Z}$ $X_{n}$ is an
$\mathcal{X}-injective$($\mathcal{X}-projective$) module.
\end{lem}

\begin{proof} Let $0 \longrightarrow N \overset{i} \longrightarrow
M$ and $\frac{M}{N} \in \mathcal{X}$ and $\alpha: N \rightarrow
X_{n}$ be linear form the pushout;

\[\begin{diagram}
\node{N} \arrow{e,t}{i} \arrow{s,t}{\alpha} \node{M} \arrow{s,t,..}{\gamma_{n}}\\
\node{X_{n}} \arrow{e,t,..}{\theta_{n}} \node{\frac{X_{n} \oplus M}{A}} \\
\end{diagram}\]

\noindent where $A=\{(\alpha(n),-i(n)): n \in N\}$.

By the following diagram;
\[\begin{diagram}
\node{0} \arrow{e} \node{X_{n + 1}} \arrow{e} \arrow{s} \node{X_{n +
1}} \arrow{e} \arrow{s} \node{0} \arrow{e} \arrow{s} \node{0} \\
\node{0} \arrow{e} \node{X_{n}} \arrow{e} \arrow{s} \node{\frac{{M
\oplus X_{n}}}{A}} \arrow{e} \arrow{s} \node{\frac{M}{N}} \arrow{e}
\arrow{s}  \node{0} \\
\node{0} \arrow{e} \node{X_{n - 1}} \arrow{e} \node{X_{n - 1}}
\arrow{e} \node{0} \arrow{e} \node{0}
\end{diagram}\]

\noindent We have the exact sequence $0 \longrightarrow X
\longrightarrow T \longrightarrow S\longrightarrow 0$ where

$T: ... \longrightarrow X_{n + 2} \longrightarrow  X_{n + 1}
\longrightarrow \frac{M \oplus X_{n}}{A} \longrightarrow X_{n - 1}
...$ and

$S: ... \longrightarrow 0 \longrightarrow  0 \longrightarrow \frac{M
}{N} \longrightarrow 0 ...$.

\noindent Since $X$ is a $\mathcal{X}-injective$ complex,
$Ext^{1}(S,X)=0$, and so

$0 \rightarrow Hom(S,X) \rightarrow Hom(T,X) \rightarrow Hom(X,X)
\rightarrow Ext^{1}(S,X)=0$. Therefore there exists $\beta_{n}:
T_{n}= \frac{M\oplus X_{n}}{A} \longrightarrow X_{n}$ such that
${\beta_{n}}{\theta_{n}}=1$. So

$${{\beta}^{n}}{\theta^{n}} (\alpha(n))=\alpha(n)$$
$${\beta}^{n}((\alpha(n),0)+A)=\alpha(n)$$
$${\beta}^{n}((0,i)+A)=\alpha(n)$$
$${{\beta}^{n}}{\gamma_{n}}i(n)=\alpha(n)$$

\noindent And hence ${{\beta}^{n}}{\gamma_{n}}i=\alpha$. So $X_{n}$
is a $\mathcal{X}-injective$ module.
\end{proof}

\begin{rem} \label{1.10} Let $X:... \rightarrow X_{n+1} \rightarrow X_{n} \rightarrow X_{n-1} \rightarrow ...$
be a complex such that $X_{n}$ are
$\mathcal{X}-injective$($\mathcal{X}-projective$) modules for all $n
\in \mathbb{Z}$. It need not to be that X is an
$\mathcal{X}-injective$($\mathcal{X}-projective$) complex.

Let R be an injective module and $f:R \rightarrow R$ be a $1-1$
morphism. Then we can find $g \neq 0$ and $g:R \rightarrow R$ such
that $gf=0$. We have the following diagram;

\[\begin{diagram}
\node{...} \arrow{e} \node{0} \arrow{e} \arrow{s} \node{R}
\arrow{e,t} {f} \arrow{s,r} {(f,0)} \node{R} \arrow{e} \arrow{s,r} {1} \node{0} \arrow{e} \arrow{s} \node{...} \\
\node{...} \arrow{e} \node{0} \arrow{e} \arrow{s} \node{R \oplus S}
\arrow{e,t} {(i,0)} \arrow{s} \node{R} \arrow{e} \arrow{s,r} {g}
\node{0} \arrow{e}
\arrow{s}  \node{...} \\
\node{...} \arrow{e} \node{0} \arrow{e} \node{0} \arrow{e} \node{R}
\arrow{e} \node{0} \arrow{e} \node{...}
\end{diagram}\]

\noindent Then $g(i,0) \neq 0$. So $\overline{R}$ cannot be an
$\mathcal{X}-injective$ complex.
\end{rem}

\noindent Dually, we can give an example for
$\mathcal{X}-projective$.

\begin{defn} \label{1.15} Let $\varepsilon$ be the class of exact complexes. Then
we can define $\varepsilon_{1}$ such that $\varepsilon_{1}$ is the
class of  exact complexes whose kernels  are in $\mathcal{X}$.
\end{defn}
\begin{lem} \label{1.16} Let $I$ be an $\mathcal{X}$-projective(injective)
module. Then $\underline{I}$ is in $\varepsilon_{1}^{\bot}(^{\bot}\varepsilon_{1})$.
\end{lem}

\begin{proof} Let
\[\begin{diagram}
\node{0} \arrow{e} \node{I} \arrow{e} \node{I^{1}} \arrow{e,t}
{\lambda_{1}} \node{I^{2}} \arrow{e,t} {\lambda_{2}} \node{I^{3}} \arrow{e} \node{...} \\
\end{diagram}\] be a projective resolution of $I$.

Then $Hom(E,I^{1}) \overset{\alpha} \rightarrow Hom(E,I^{2})
\overset{\beta} \rightarrow Hom(E,I^{3})$ where  $E:... \rightarrow
E^{-2} \rightarrow E^{-1} \rightarrow E^{0} \rightarrow E^{1}
\rightarrow...$\\

$Ext^{1}(E,I)=?$

\[\begin{diagram}
\node{E:...} \arrow{e} \node{E^{-3}} \arrow{e,t} {d^{-3}} \arrow{s}
\node{E^{-2}} \arrow{e,t} {d^{-2}} \arrow{s,r} {u^{-2}}
\node{E^{-1}} \arrow{e,t} {d^{-1}} \arrow{s,r} {u^{-1}} \node{E^{0}}
\arrow{e} \arrow{s} \node{...} \\
\node{I^{2}:...} \arrow{e} \node{0} \arrow{e} \arrow{s} \node{I^{0}}
\arrow{e,t} {1} \arrow{s,r} {1} \node{I^{0}} \arrow{e} \arrow{s} \node{0} \arrow{e} \arrow{s} \node{...} \\
\node{I^{3}:...} \arrow{e} \node{I^{0}} \arrow{e} \node{I^{0}}
\arrow{e} \node{0} \arrow{e} \node{0} \arrow{e} \node{...}\\
\end{diagram}\]

\noindent where ${{\lambda_{2}}^{-2}}{u^{-2}}=0$, since
${\lambda_{2}}^{-2}=1$, then $u^{-2}=0$ and also since
${u^{-2}}={u^{-1}}{d^{-2}}$, then $0={u^{-1}}{d^{-2}}$.

\[\begin{diagram}
\node{E:...} \arrow{e} \node{E^{-3}} \arrow{e,t} {d^{-3}} \arrow{s}
\node{E^{-2}} \arrow{e,t} {d^{-2}} \arrow{s,r} {u^{-2}}
\node{E^{-1}} \arrow{e,t} {d^{-1}} \arrow{s,r} {f^{-1}} \node{E^{0}}
\arrow{e} \arrow{s,r} {f^{0}}
\node{E^{1}} \arrow{e} \arrow{s} \node{...} \\
\node{I^{1}:...} \arrow{e} \node{0} \arrow{e} \arrow{s} \node{0}
\arrow{e} \arrow{s} \node{I^{0}}
\arrow{e,t} {1} \arrow{s,r} {1} \node{I^{0}} \arrow{e} \arrow{s} \node{0} \arrow{e} \arrow{s} \node{...} \\
\node{I^{2}:...} \arrow{e} \node{0} \arrow{e} \node{I^{0}}
\arrow{e} \node{I^{0}} \arrow{e} \node{0} \arrow{e} \node{0} \arrow{e} \node{...}\\
\end{diagram}\]

\noindent where ${f^{0}}{d^{-1}}={f^{-1}}$ and ${f^{-1}}{d^{-2}}=0$.

Since $Kerd^{-1} \overset{i} \rightarrow E^{-1} \overset{d^{-1}}
\rightarrow E^{0}$ is exact and $I^{0}$ is $\mathcal{X}$-injective,
then $Hom(E^{0},I^{0}) \overset{(d^{-1})^{*}} \rightarrow
Hom(E^{-1},I^{0}) \overset{i^{*}} \rightarrow Hom(Kerd^{-1},I^{0})$
is exact.

\end{proof}

\begin{cor} \label{1.17} Let $I=...\rightarrow 0 \rightarrow I_{n} \rightarrow I_{n-1} \rightarrow ... \rightarrow I_{0}
\rightarrow 0 \rightarrow ...$ where $I_{i}$ is an
$\mathcal{X}-projective(injective)$ module, then $I$ is in
${\varepsilon_{1}}^{\bot}(^{\bot}\varepsilon_{1})$.
\end{cor}

\begin{proof} Since ${\varepsilon_{1}}^{\bot}$ is extension closed,
by Lemma \ref{1.16} we can understand that $I$ is in
${\varepsilon_{1}}^{\bot}$.
\end{proof}

\begin{cor}Every left(right) bounded complex I where $I_{i}$ is an
$\mathcal{X}-projective(injective)$ module is in
${\varepsilon_{1}}^{\bot}(^{\bot}\varepsilon_{1})$.

\end{cor}

\begin{lem} \label{1.11} If $I \in \varepsilon_{1}^\bot$, then each
$I^{n}$ is $\mathcal{X}-injective$ for each $n \in \mathbb{Z}$.
\end{lem}

\begin{proof} Let $S \subseteq M$ be a submodule of a module M with $\frac{M}{S} \in \mathcal{X}$ and $\alpha: S
\longrightarrow I_{n}$ be linear form the pushout;

\[\begin{diagram}
\node{S} \arrow{e,t}{i} \arrow{s,t}{\alpha} \node{M} \arrow{s,t,..}{i_{1}}\\
\node{I^{n}} \arrow{e,t,..}{i_{2}} \node{\frac{I^{n}\oplus M}{A}=I^{n}\oplus_{S} M}\\
\end{diagram}\]

\noindent where $A=\{(\alpha(s),-s): s\in S\}$. Thus $i_{2}$ is
one-to-one the same as i. Then $\overline{I}:... \longrightarrow
{I^{n-1}} \longrightarrow {I^{n}\oplus_{S} M} \longrightarrow
{I^{n+1}} \longrightarrow {I^{n+2}} \longrightarrow ...$ is a
complex.

\[\begin{diagram}
\node{0} \arrow{e} \node{I^{n-1}} \arrow{e} \arrow{s} \node{I^{n-1}}
\arrow{e} \arrow{s} \node{0} \arrow{e} \arrow{s} \node{0} \\
\node{0} \arrow{e} \node{I^{n}} \arrow{e} \arrow{s} \node{I^{n}
\oplus_{S} M} \arrow{e} \arrow{s} \node{\frac{M}{S}} \arrow{e}
\arrow{s}  \node{0} \\
\node{0} \arrow{e} \node{I^{n+1}} \arrow{e} \node{I^{n+1}} \arrow{e}
\node{\frac{M}{S}} \arrow{e} \node{0} \\
\end{diagram}\]

Therefore, we have an exact sequence $0 \longrightarrow I
\longrightarrow \overline{I} \longrightarrow E \longrightarrow 0$
where $E:...\longrightarrow \frac{M}{S} \longrightarrow \frac{M}{S}
\longrightarrow 0 \longrightarrow 0 \longrightarrow 0
\longrightarrow ...$ and so we have an exact sequence $0
\longrightarrow Hom(E,I) \longrightarrow Hom(\overline{I},I)
\longrightarrow Hom(I,I) \longrightarrow Ext'(E,I)=0$ since $I \in
\varepsilon_{1}^\bot$.

This implies that we can find $\overline{f}:\overline{I}
\longrightarrow I$ with $\overline{f}f=1$. Therefore, there exists a
function $\overline{f}^{n}:{I^{n} \oplus_{S} M} \longrightarrow
I^{n}$ with $\overline{f}^{n}f^{n}=1$. So,

$$\overline{f}^{n}f^{n} (\alpha(s))=\alpha(s)$$
$$\overline{f}^{n}((\alpha(s),0)+A)=\alpha(s)$$
$$\overline{f}^{n}((0,s)+A)=\alpha(s)$$
$$\overline{f}^{n}i_{1}i(s)=\alpha(s)$$

\noindent and hence ${\overline{f}_{n}}i_{1}i=\alpha$ and thus each
$I^{n} \in \mathcal{X}-injective$.

\end{proof}

\begin{lem} \label{1.12} Let $f:X \longrightarrow Y$ be a morphism of complexes.
Then the exact sequence $0 \longrightarrow Y \longrightarrow M(f)
\longrightarrow X[1] \longrightarrow 0$ associated with the mapping
cone $M(f)$ splits in $\mathcal{C(X)}$ if and only if f is homotopic
to 0.
\end{lem}

\begin{proof} It follows from the proof by \cite{3}.
\end{proof}

\begin{lem} \label{1.13} Let $X$ and $I$ are complexes. If $Ext^{1}(X,I[n])=0$
for all $n \in \mathbb{Z}$, then $\mathcal{H}om(X,I)$ is exact.
\end{lem}

\begin{proof} Since $Ext^{1}(X,I(n))=0$, if $f: X[-1] \rightarrow
I[n]$ is a morphism, then $0 \rightarrow I[n] \rightarrow M(f)
\rightarrow X \rightarrow 0$ splits.

By \ref{1.12}, $f: X[-1] \rightarrow I[n]$ is homotopic to zero for
all n.

So $f^{1}: X \rightarrow I[n+1]$ is homotopic to zero for all $n \in
\mathbb{Z}$. Thus $\mathcal{H}om(X,I)$ is exact.
\end{proof}

\begin{cor} ${\varepsilon_{1}}^{\bot}(^{\bot}\varepsilon_{1}) \subseteq
DG(\mathcal{X}-projective(injective))$.
\end{cor}

\begin{proof} It follows from Lemma 2.13 and 2.15.
\end{proof}
\begin{cor}Every left(right) bounded complex I where $I_{i}$ is an
$\mathcal{X}-projective(injective)$ module is in
$DG(\mathcal{X}-projective(injective))$.

\end{cor}

\begin{proof} We can understand by Corollary 2.12 and
Corollary 1.16.
\end{proof}

\begin{lem} \label{1.14} Let $X$ be extension closed and let $I$ be a
DG($\mathcal{X}$-injective)(DG($\mathcal{X}$)-projective)
complex($proj \in \mathcal{X}$?). Then $Ext^{1}(E,I[n])=0$ for all
$n \in \mathbb{Z}$ where $E$ is exact and $Ker\lambda_{1} \in
\mathcal{X}$ with $\lambda_{n}: E_{n} \rightarrow E_{n-1}$.
\end{lem}

\begin{proof} Let $... \longrightarrow P_{2} \overset{f_{2}} \longrightarrow
P_{1} \overset{f_{1}} \longrightarrow P_{0}
 \overset{f_{0}} \longrightarrow E \longrightarrow 0$
be a projective resolution of E. Then we have

$$Hom(P_{0},I) \overset{{f_{1}}^{*}}
\longrightarrow Hom(P_{1},I) \overset{{f_{2}}^{*}} \longrightarrow
Hom(P_{2},I)$$.

Let ${\beta} \in Ker{{f_{2}}^{*}}$. Is ${\beta} \in
Im{{f_{1}}^{*}}?$

Since ${P_{2}}^{n} \longrightarrow {P_{1}}^{n} \longrightarrow
{P_{0}}^{n}$ is exact and $E^{n} \in \mathcal{X}$ and $I^{n}$ is
$\mathcal{X}-injective$ by Lemma 2.6

$$Hom({P_{0}}^{n},I^{n}) \overset{{{f_{1}}^{n}}^{*}}
\longrightarrow Hom({P_{1}}^{n},I^{n}) \overset{{f_{2}}^{*}}
\longrightarrow Hom({P_{2}}^{n},I^{n})$$ is exact.

${{f_{2}}^{*}}{\beta}=0$ implies that ${\beta}{f_{2}}=0$

$$\Rightarrow {{\beta}^{n}}{{f_{2}}^{n}}=0$$, for all $n \in \mathbb{Z}$

$$\Rightarrow {{f_{2}}^{*}}{{\beta}^{n}}=0$$

$$\Rightarrow {{\beta}^{n}} \in Ker{{f_{2}}^{*}}=Im{{f_{1}}^{*}}$$

$$\Rightarrow \exists {\theta^{n}} \in Hom(P^{0},I^{n})$$ such that
$${{\theta}^{n}}{{f_{1}}^{n}}={{\beta}^{n}}$$

Is $\theta$ a chain map?

We know that $\mathcal{H}om(E,I)$ is exact,

\begin{enumerate}[{\bf(i)}] \item Is $\mathcal{H}om(P_{0},I)$ exact?(where $I$ is
$\mathcal{X}-injective$ and $P_{0} \longrightarrow E \longrightarrow
0$ and $E^{n} \in \mathcal{X}$)

\noindent \item Let $\mathcal{H}om(P_{0},I)$ be exact. We have
${\theta^{n}}: {P_{0}}^{n} \longrightarrow I^{n}$. Is it necessary
$\theta$ is a chain map?
\end{enumerate}

\[\begin{diagram}
\node{P_{0}^{n-1}} \arrow{e,t} {\lambda^{n-1}} \arrow{s,r}
{\theta^{n-1}} \node{P_{0}^{n}} \arrow{e,t} {\lambda^{n}}
\arrow{s,r} {\theta^{n}} \arrow{sw,..} \node{P_{0}^{n +
1}} \arrow{s,r} {\theta^{n+1}} \arrow{sw,..} \\
\node{I^{n-1}} \arrow{e,t} {\gamma^{n-1}} \node{I^{n}} \arrow{e,t}
{\gamma^{n}} \node{I^{n+1}} \\
\end{diagram}\]

\[\begin{diagram}
\node{P_{0}^{n-1}} \arrow{e,t} {\lambda^{n-1}} \arrow{s,r} {t^{n-1}}
\node{P_{0}^{n}} \arrow{e,t} {\lambda^{n}} \arrow{s,r} {t^{n}}
\arrow{sw,r,..} {s^{n}}
\node{P_{0}^{n+1}} \arrow{s,r} {t^{n+1}} \arrow{sw,r,..} {s^{n+1}} \\
\node{I^{n}} \arrow{e,t} {\gamma^{n}} \node{I^{n+1}} \arrow{e,t}
{\gamma^{n+1}} \node{I^{n+2}} \\
\end{diagram}\]

where
$t^{n-1}={{\theta^{n}}{\lambda^{n-1}}-{\gamma^{n-1}}{\theta^{n-1}}}$
and  $t^{n}={{\theta^{n+1}}{\lambda^{n}}-{\gamma^{n}}{\theta^{n}}}$

Since $\mathcal{H}om(P_{0},I[n])$ is exact, we have a homotopy such
that
$${s^{n+1}}{\lambda^{n}}+{\gamma^{n}}{s^{n}}={\theta^{n+1}}{\lambda^{n}}-{\gamma^{n}}{\theta^{n}}$$
$$\gamma^{n}(s^{n}+\theta^{n})=(\theta^{n+1}+s^{n+1})\lambda^{n}$$\\
So we have a chain map. But we investigate a chain map such that
$${{\theta}^{n}}{{f_{1}}^{n}}={{\beta}^{n}}$$. How can we do?

\end{proof}

\begin{lem} \label{1.19} Let $f:X \longrightarrow Y$ a chain morphism, $Y$ is an
$\mathcal{X}$ complex and X is an $\mathcal{X}-projective$ complex.
Then $f$ is a homotopic to zero.
\end{lem}

\begin{proof} Let $id:Y \longrightarrow Y$ and the exact sequence $0 \longrightarrow Y[-1] \longrightarrow M(id)[-1]
\longrightarrow Y \longrightarrow 0$.

Since $X$ is an $\mathcal{X}-projective$ complex, we have the
following commutative diagram;

\[\begin{diagram}
\node{M(id)[-1]} \arrow{e,r} {\pi} \node{Y} \arrow{e} \node{0}\\
\node{X} \arrow{n,r,..} {g} \arrow{ne,t} {f}
\end{diagram}\]

\noindent where ${\pi}{g}={f}$. Let $\pi':M(id)[-1] \longrightarrow
Y[-1]$ be a projection. Then if we take an $s={\pi'}{g}$, then for
all $n\in \mathbb{Z}$,
${s^{n+1}}{\lambda^{n}}+{\gamma^{n-1}}{s^{n}}={f^{n}}$ where
$\lambda$ and $\gamma$ are boundary maps of the complexes of $X$ and
$Y$, respectively. So $f$ is homotopic to zero.
\end{proof}

\begin{lem} \label{1.20} Let $f:X \longrightarrow Y$ be a chain
morphism, $X$ is an $\mathcal{X}-complex$ and $Y$ is an
$\mathcal{X}-injective$ complex. Then $f$ is homotopic to zero.
\end{lem}

\begin{proof}Let $id:X \longrightarrow X$, then we have the
following exact sequence;

\[\begin{diagram}
\node{0} \arrow{e} \node{X} \arrow{e,t} {i} \arrow{s,r} {f}
\node{M(id)} \arrow{e} \arrow{sw,r,..} {g} \node{X[1]} \arrow{e}
\node{0}\\
\node[2]{Y}
\end{diagram}\]

\noindent where ${g}{i}={f}$.

Let $i':X[1] \longrightarrow M(id)$ and $s:X[1] \longrightarrow Y$
such that ${s}={g}{i'}$ with ${s_{n}}={g^{n-1}}{i'}$.

\[\begin{diagram}
\node{{X^{n-2}}\oplus {X^{n-1}}} \arrow{e,t} {u^{n-2}} \arrow{s,r}
{g^{n-2}} \node{{X^{n-1}}\oplus {X^{n}}} \arrow{e,t} {u^{n-1}}
\arrow{s,r} {g^{n-1}} \node{{X^{n}}\oplus {X^{n+1}}} \arrow{s,r}
{g^{n}}\\
\node{Y^{n-2}} \arrow{e,t} {\gamma^{n-2}} \node{Y^{n-1}} \arrow{e,t} {\gamma^{n-1}} \node{Y^{n}}\\
\end{diagram}\]

${s^{n+1}}{\lambda^{n}}+{\gamma^{n-1}}{s^{n}}={g^{n}}{i'}{\lambda^{n}}+{\gamma^{n-1}}{g^{n-1}}{i'}
={g^{n}}({i'}{\lambda^{n}}+{g^{n}}{u^{n-1}}{i'}={g^{n}}{{i'}{\lambda^{n}}+{u^{n-1}}{i'}})
={g^{n}}{i}={f^{n}}$
\end{proof}

\begin{prop} \label{1.22} Let $C \longrightarrow X$ and $C' \longrightarrow X$ be
$\mathcal{X}-projective$ covers of $X \in \mathcal{X}$, then $C$ and
$C'$ are homotopic.
\end{prop}

\begin{proof}
It follows from \cite{4}.
\end{proof}

\begin{lem} \label{1.23} Let $X:...\longrightarrow X_{2} \longrightarrow X_{1} \longrightarrow
X_{0} \longrightarrow 0$ be an exact and $P$ be a complex with for
all $n \geq 0$, $P_{n}$ $\mathcal{X}-projective$ module, then $f:P
\longrightarrow X$ is a homotopy.
\end{lem}

\begin{proof} It follows from \cite{4}.
\end{proof}

\begin{lem} \label{1.24} Let every R-module has an onto $\mathcal{X}$-projective $\mathcal{X}$
precover with kernel in $\mathcal{X}$. Then every bounded complex has an $C(\mathcal{X}-projective)$ precover.
\end{lem}

\begin{proof} Let $Y(n): ... \rightarrow 0 \rightarrow Y^{0} \rightarrow Y^{1} \rightarrow ... \rightarrow Y^{n} \rightarrow
0 \rightarrow...$.

We will use induction on n. Let $n=0$, then we have the following
commutative diagram;

\[\begin{diagram}
\node{D(0):...} \arrow{e} \node{0} \arrow{e} \arrow{s} \node{P^{0}}
\arrow{e,t} {id} \arrow{s,r} {f^{0}} \node{P^{0}} \arrow{e}
\arrow{s} \node{0} \arrow{e}
\arrow{s} \arrow{e} \node{...} \\
\node{Y(0):...} \arrow{e} \node{0} \arrow{e} \node{Y^{0}} \arrow{e}
\node{0} \arrow{e} \node{0}
\arrow{e} \node{...} \\
\end{diagram}\]

\noindent where $D(0)$ is exact and $Ker(D(0) \rightarrow Y(0)) \in
\mathcal{X}$.

Let $n=1$, then we have the following commutative diagram;

\[\begin{diagram}
\node{D(1):...} \arrow{e} \node{0} \arrow{e} \node{P^{0}}
\arrow{e,t} {\lambda_{1}^{0}} \arrow{s,r} {f^{0}} \node{P^{0} \oplus
P^{1}} \arrow{e,t} {\lambda^{1}} \arrow{s,r} {(0,f^{1})}
\node{P^{1}} \arrow{e}
\arrow{s} \node{0} \arrow{e} \node{...} \\
\node{Y(1):...} \arrow{e} \node{0} \arrow{e} \node{Y^{0}}
\arrow{e,t} {a^{0}} \node{Y^{1}} \arrow{e} \node{0} \arrow{e}
\node{0} \arrow{e} \node{...} \\
\end{diagram}\]

\noindent where $D(1)$ is exact and $Ker(D(1) \rightarrow Y(1)) \in
\mathcal{X}$.

We also have the following commutative diagram;
\[\begin{diagram}
\node{P^{0}} \arrow{e,t} {s^{0}} \arrow{s,r} {f^{0}} \node{P^{1}}
\arrow{s,r} {f^{1}} \\
\node{Y^{0}} \arrow{e,t} {a^{0}} \node{Y^{1}} \\
\end{diagram}\]

$\lambda_{1}^{0}(x)=(x,s^{0}(x))$ and $\lambda^{1}(x,y)=s^{0}(x)-y$

We assume that the following diagram which is commutative;

\[\begin{diagram}
\node{D(n):...0} \arrow{e} \node{P^{0}} \arrow{e,t} {\lambda^{0}}
\arrow{s,r} {f^{0}} \node{P^{0} \oplus P^{1}} \arrow{e,t}
{\lambda^{1}} \arrow{s,r} {(0,f^{1})} \node{...} \arrow{e}
\node{P^{n-1} \oplus P^{n}} \arrow{e,t} {\lambda^{n}}
\arrow{s,r} {(0,f^{n})} \node{P^{n}} \arrow{e} \arrow{s} \node{0} \\
\node{Y(n):...0} \arrow{e} \node{Y^{0}} \arrow{e,t} {a^{0}}
\node{Y^{1}} \arrow{e,t} {a^{1}} \node{...} \arrow{e} \node{Y^{n}}
\arrow{e} \node{0}  \arrow{e} \node{0}\\
\end{diagram}\]

\noindent where $D(n)$ is exact and $Ker(D(n) \rightarrow Y(n)) \in
\mathcal{X}$.

And we also have the following commutative diagrams;

\[\begin{diagram}
\node{D(n)} \arrow{e,t} {s} \arrow{s} \node{\overline{P^{n+1}}}
\arrow{s} \\
\node{Y(n)} \arrow{e} \node{\underline{Y^{n+1}}} \\
\end{diagram}\]

So,
\[\begin{diagram}
\node{D(n):...0} \arrow{e} \node{P^{0}} \arrow{e,t} {\lambda^{0}}
\arrow{s} \node{P^{0} \oplus P^{1}} \arrow{e,t} {\lambda^{1}}
\arrow{s} \node{...} \arrow{e,t} {\lambda^{n-1}} \node{P^{n-1}
\oplus P^{n}} \arrow{e,t} {\lambda^{n}}
\arrow{s,r} {s^{1}} \node{P^{n}} \arrow{e} \arrow{s,r} {s^{2}} \node{0} \\
\node{\overline{P^{n+1}}:...0} \arrow{e} \node{0} \arrow{e} \node{0}
\arrow{e} \node{...} \arrow{e} \node{P^{n+1}}
\arrow{e,t} {1} \node{P^{n+1}} \arrow{e} \node{0}\\
\end{diagram}\]

\noindent where ${s^{2}}{\lambda^{n}}={s^{1}}$ and
${s^{1}}{\lambda^{n-1}}=0$.

\[\begin{diagram}
\node{P^{n-1} \oplus P^{n}} \arrow{e,t} {s^{1}} \arrow{s,r}
{(0,f^{n})} \node{P^{n+1}}
\arrow{s,r} {f^{n+1}} \\
\node{Y^{n}} \arrow{e,t} {a^{n}} \node{Y^{n+1}} \\
\end{diagram}\]

\noindent where ${f^{n+1}}{s^{1}}={a^{n}}{(0,f^{n})}$.

\[\begin{diagram}
\node{P^{n}} \arrow{e,t} {s^{2}} \arrow{s} \node{P^{n+1}}
\arrow{s,r} {f^{n+1}} \\
\node{0} \arrow{e} \node{Y^{n+1}} \\
\end{diagram}\]

\noindent where ${f^{n+1}}{s^{2}}=0$.

\[\begin{diagram}
\node{D(n+1):...0} \arrow{e} \node{P^{0}} \arrow{e,t} {\lambda^{0}}
\arrow{s,r} {f^{0}} \node{P^{0} \oplus P^{1}...} \arrow{e}
 \arrow{s,r} {(0,f^{1})} \node{P^{n-1} \oplus P^{n}}
\arrow{e,t} {\lambda^{n}_{1}}
\arrow{s,r} {(0,f^{n})} \node{P^{n} \oplus P^{n+1}} \arrow{e,t} {\lambda^{n+1}} \arrow{s,r} {(0,f^{n+1})} \node{P^{n+1}...}\\
\node{Y(n+1):...0} \arrow{e} \node{Y^{0}} \arrow{e,t} {a^{0}}
\node{Y^{1}...} \arrow{e,t} {a^{1}} \node{Y^{n}}
\arrow{e,t} {a^{n}} \node{Y^{n+1}} \arrow{e} \node{0...}\\
\end{diagram}\]

\noindent where $D(n+1)$ is exact and $Ker(D(n+1) \rightarrow
Y(n+1)) \in \mathcal{X}$,\;
$\lambda_{1}^{n}(x,y)=(\lambda^{n}(x,y),s^{1}(x,y))$, \;
$\lambda^{n+1}(x,y)=s^{2}(x)-y$.

Therefore, $Y(n)$ has an $C(\mathcal{X})$-projective precover.

\end{proof}

\begin{lem} \label{1.25} Let every R-module has an epic
$\mathcal{X}$-injective $\mathcal{X}$ preenvelope with cokernel in $\mathcal{X}$.
Then every bounded complex  has an
$C(\mathcal{X}-injective)$ preenvelope.
\end{lem}

\begin{proof} Let $Y(n): ... \rightarrow 0 \rightarrow Y^{0} \rightarrow Y^{1} \rightarrow ... \rightarrow Y^{n} \rightarrow
0 \rightarrow...$.

We will use induction on n. Let $n=0$, then we have the following
commutative diagram;

\[\begin{diagram}
\node{Y(0):...} \arrow{e} \node{0} \arrow{e} \arrow{s} \node{0}
\arrow{e} \arrow{s} \node{Y^{0}} \arrow{e} \arrow{s} \node{0}
\arrow{e} \arrow{s} \node{0}
\arrow{e} \arrow{s} \node{...} \\
\node{E(0):...} \arrow{e} \node{0} \arrow{e} \node{E_{0}}
\arrow{e,t} {id} \node{E_{0}} \arrow{e} \node{0} \arrow{e} \node{0} \arrow{e} \node{...} \\
\end{diagram}\]

\noindent where $E(0)$ is exact and cokernel in $\mathcal{X}$.

Let $n=1$, then we have the following commutative diagram;

\[\begin{diagram}
\node{Y(1):...0} \arrow{e} \node{0} \arrow{e} \arrow{s} \node{Y_{1}}
\arrow{e,t} {a_{1}} \arrow{s,r} {(f_{1},0)} \node{Y_{0}} \arrow{e}
\arrow{s,r} {f_{0}} \node{0} \arrow{e} \arrow{s} \node{...} \\
\node{E(1):...0} \arrow{e} \node{E_{1}} \arrow{e,t} {\lambda_{2}}
\node{E_{1} \oplus E_{0}} \arrow{e,t}
{\lambda_{1}} \node{E_{0}} \arrow{e} \node{0} \arrow{e} \node{...} \\
\end{diagram}\]

\noindent where $E(1)$ is exact and cokernel in $\mathcal{X}$.

We also have the following commutative diagram;
\[\begin{diagram}
\node{\underline{Y_{1}}} \arrow{e} \arrow{s}
\node{\underline{Y_{0}}}
\arrow{s} \\
\node{\overline{E_{1}}} \arrow{e} \node{\overline{E_{0}}} \\
\end{diagram}\]

\noindent where $E_{1} \overset{\lambda} \rightarrow E_{0}$ such
that ${f_{0}}{a_{1}}={\lambda}{f_{1}}$.

$\lambda_{2}(x)=(x,-f(x))$ and $\lambda_{1}(x,y)=f(x)+y$.

We assume that the following diagram which is commutative;

\[\begin{diagram}
\node{Y(n):...0} \arrow{e}  \node{0}  \arrow{e} \arrow{s}
\node{Y_{n}} \arrow{e,t} {a_{n}} \arrow{s,r} {(f_{n},0)} \node{...}
\arrow{e} \node{Y_{0}}
\arrow{e} \arrow{s,r} {f_{0}} \node{0}\\
\node{E(n):...0} \arrow{e} \node{E_{n}} \arrow{e,t} {\lambda_{n}}
\node{E_{n} \oplus E_{n-1}} \arrow{e,t} {\lambda_{n-1}} \node{...}
\arrow{e}
\node{E_{0}} \arrow{e} \node{0} \\
\end{diagram}\]

\noindent where $E(n)$ is exact and cokernel in $\mathcal{X}$.

And we have the following commutative diagram;

\[\begin{diagram}
\node{Y(n+1):...0} \arrow{e} \node{0} \arrow{e} \arrow{s}
\node{Y_{n+1}} \arrow{e,t} {a_{n+1}} \arrow{s,r} {(f_{n+1},0)}
\node{Y_{n}...} \arrow{e,t} {a_{n}} \arrow{s,r} {(f_{n},0)}
\node{Y_{0}} \arrow{e}
\arrow{s,r} {f^{0}} \node{0}\\
\node{E(n+1):...0} \arrow{e} \node{E_{n+1}} \arrow{e,t}
{\lambda_{n+1}} \node{E_{n+1} \oplus E_{n}} \arrow{e,t}
{\lambda_{n}} \node{E_{n} \oplus E_{n-1}...} \arrow{e} \node{E_{0}} \arrow{e} \node{0...}\\
\end{diagram}\]

\noindent where $E(n+1)$ is exact and cokernel in $\mathcal{X}$.

We also have the following commutative diagrams;

\[\begin{diagram}
\node{\underline{Y_{n+1}}} \arrow{e} \arrow{s} \node{Y(n)} \arrow{s} \\
\node{\overline{E_{n+1}}} \arrow{e} \node{E(n)}\\
\end{diagram}\]

\[\begin{diagram}
\node{Y_{n+1}} \arrow{e,t} {a_{n+1}} \arrow{s,r}
{f_{n+1}} \node{Y_{n}} \arrow{s,r} {(f_{n},0)}\\
\node{E_{n+1}} \arrow{e,t} {s_{n+1}} \node{E_{n} \oplus E_{n-1}} \\
\end{diagram}\]

\noindent where $(f_{n},0){a_{n+1}}={s_{n+1}}{f_{n+1}}$.

\[\begin{diagram}
\node{E_{n+1}} \arrow{e,t} {t_{n+1}} \arrow{s,r} {1}
\node{E_{n}} \arrow{s,r} {\lambda_{n}^{1}}\\
\node{E_{n+1}} \arrow{e,t} {s_{n+1}} \node{E_{n} \oplus E_{n-1}} \\
\end{diagram}\]

\noindent where ${\lambda_{n}^{1}}{t_{n+1}}={s_{n+1}}{1}$.

Then, \; $\lambda_{n+1}(x)=(x,-f_{n+1}(x))$, \;
$\lambda_{n}(x,y)=s_{n+1}(x)+\lambda^{1}_{n}(y)$.

Therefore, $Y(n)$ has an $C(\mathcal{X})$-injective preenvelope.

\end{proof}

\begin{lem} \label{1.27} Let $X$ be an $\mathcal{X}-injective$ complex and $\frac{E(X)}{X} \in \mathcal{X}$
(or $\frac{Y}{X} \in \mathcal{X}$) where $E(X)$ is an injective
envelope of $X$. Then $X=E(X)$ and so it is an exact complex.(X is a
direct summand of Y).
\end{lem}

\begin{proof} We know that every complex has an injective envelope,
so $X$ has an injective envelope $E(X)$. Then $E(X)$ is a injective
complex, and so it is exact. We have the following commutative
diagram;

\[\begin{diagram}
\node{0} \arrow{e}  \node{X} \arrow{s,t} {id_x} \arrow{e,r} {i}
\node{E(X)} \arrow{sw,r,..} {\phi} \\
\node[2]{X}
\end{diagram}\]

\noindent such that ${\phi}{i}={id}_{x}$. Therefore $X$ is a direct
summand of $E(X)$. So $X$ is an injective complex and hence it is
exact. Similarly, if $\frac{Y}{X} \in \mathcal{X}$, then we can
prove that X is a direct summand of Y.
\end{proof}

\begin{thm} \label{1.28} Let $\mathcal{X}$ be a class under extensions, quotients and direct
limits, then every complex $B$ has an $\mathcal{X}-injective$
envelope.
\end{thm}

\begin{proof} We know that $B$ has an injective envelope $E$. Let
$S= \{A: B \subseteq A \subseteq E$ $and$ $\frac{A}{B}$ $is$ $ an$
$\mathcal{X}-complex \} \neq \emptyset$. Let $S'$ be an ascending
chain. Then $\frac{\cup_{N_{i}\in
S'}N^{k}_{i}}{B^{k}}=\cup_{N_{i}\in
S'}\frac{N^{k}_{i}}{B^{k}}=lim\frac{N^{k}_{i}}{B^{k}}$ is in the
class $\mathcal{X}$. So $S$ has a maximal element, say $T$. We shall
prove that $T$ is an $\mathcal{X}-injective$ complex. It is enough
to show that any exact sequence $0 \longrightarrow T \longrightarrow
Y \longrightarrow C \longrightarrow 0$ with $C \in
\mathcal{X}$-complex is split. We have the following commutative
diagram;

\[\begin{diagram}
\node{0} \arrow{e} \node{T} \arrow{e} \arrow{s} \node{Y} \arrow{e}
\arrow{s,r} {\alpha} \node{C} \arrow{e}
\arrow{s,r} {\beta} \node{0} \\
\node{0} \arrow{e} \node{T} \arrow{e,t} {i} \node{E(T)} \arrow{e,t}
{\pi} \node{\frac{E(T)}{T}} \arrow{e}
\node{0} \\
\end{diagram}\]

Let $\beta(C)=\frac{K}{T}$. So by
$\frac{K}{T}\cong\frac{\frac{K}{B}}{\frac{T}{B}}$, we say that
$\frac{K}{B}$ is an $\mathcal{X}-complex$. Since $T$ is a maximal
element of $S$, $\beta(C)=0$, and hence
${\pi}{\alpha(Y)}={\frac{\alpha(Y)+{T}}{T}}=0$, and so
$\alpha(Y)\subseteq T$. Therefore $0 \longrightarrow T
\longrightarrow Y \longrightarrow C \longrightarrow 0$ is split
exact and hence $T$ is an $\mathcal{X}-injective$ complex. Moreover
$T$ is an $\mathcal{X}-injective$ $preenvelope$ of $B$ by the
following diagram;

\[\begin{diagram} \node{0} \arrow{e} \node{B} \arrow{e} \arrow{s}
\node{T} \arrow{sw,..} \\ \node[2]{Y}
\end{diagram}\]

\noindent where $Y$ is an $\mathcal{X}-injective$ complex and
$\frac{T}{B} \in \mathcal{X}$ and also if,

\[\begin{diagram} \node{0} \arrow{e} \node{B} \arrow{e,t} {\alpha}
\arrow{s,r} {\alpha} \node{T} \arrow{sw,t,..} {\beta} \\
\node[2]{T}
\end{diagram}\]

\noindent where ${\beta}{\alpha}={\alpha}$, then $\beta$ is $1-1$
since $B \subset^{ess}T$, and so $\beta(T) \cong T$. If $\beta(T)
\neq T$, then $\beta(T)$ is an $\mathcal{X}-injective$ complex such
that $B \subseteq \beta(T) \subset T \subseteq E$. So we have
$0\longrightarrow T \overset{\beta} \longrightarrow T
\longrightarrow \frac{T}{\beta(T)} \longrightarrow 0$ split exact
sequence and $\beta(T) \subset^{ess} T$, so $\beta(T)=T$.
\end{proof}

\begin{thm} \label{1.29} $X$ is essentially contained in a minimial $\mathcal{X}-injective$ complex X'
with $X \subset^{ess} X'$.
\end{thm}

\begin{proof} We know that every complex has an
injective envelope. Let $S= \{A: X \subseteq A \subseteq E$ $and$
$A$ $\mathcal{X}- injective$ $complex\} \neq \emptyset$ and $S'$ be
a descending chain of $S$. We will show that $\cap_{A_{\alpha} \in
S'}\{A_{\alpha}: \alpha \in I\}$ is an $\mathcal{X}-injective$
complex. Using this pushout diagram we have the following diagram
where $C$ with $\mathcal{X}-complex$,

\[\begin{diagram}
\node{0} \arrow{e} \node{\cap A_{\alpha}} \arrow{e,t} {\beta}
\arrow{s,r} {\theta_{\alpha}} \node{Y} \arrow{e} \arrow{s,r} {\phi}
\node{C} \arrow{e}
\arrow{s} \node{0} \\
\node{0} \arrow{e} \node{A_{\alpha}} \arrow{e,t} {\gamma_{\alpha}}
\node{B_{\alpha}} \arrow{e} \node{C} \arrow{e}
\node{0} \\
\end{diagram}\]

\noindent Then the bottom row is split exact. So $0 \longrightarrow
\cap A_{\alpha}  \longrightarrow \cap B_{\alpha} \longrightarrow C
\longrightarrow 0$ is split exact. We have the following diagram;

\[\begin{diagram}
\node{0} \arrow{e} \node{\cap A_{\alpha}} \arrow{e,t} {\beta}
\arrow{s} \node{Y} \arrow{e} \arrow{s,r} {\phi} \node{C} \arrow{e}
\arrow{s} \node{0} \\
\node{0} \arrow{e} \node{\cap A_{\alpha}} \arrow{e,t} {\gamma}
\node{\cap B_{\alpha}} \arrow{e} \node{C} \arrow{e}
\node{0} \\
\end{diagram}\]

\noindent By five lemma $\phi$ is an isomorphism. So
$0\longrightarrow \cap A_{\alpha}  \longrightarrow Y \longrightarrow
C \longrightarrow 0$ is split exact. So $S$ has a minimal element,
say $X'$.
\end{proof}

\end{document}